\documentclass[11pt]{amsart}
\usepackage{graphicx}
\usepackage{amssymb, amsmath}
\newtheorem{theorem}{Theorem}[section]
\newtheorem{corollary}[theorem]{Corollary}

\newtheorem{proposition}[theorem]{Proposition}
\theoremstyle{definition}

\theoremstyle{remark}

\begin{document}
\title{On logically cyclic  groups}
\author{\sc M. Shahryari}
\thanks{}
\address{ Department of Pure Mathematics,  Faculty of Mathematical
Sciences, University of Tabriz, Tabriz, Iran }

\email{mshahryari@tabrizu.ac.ir}
\date{\today}

\begin{abstract}
A group $G$ is called logically cyclic, if it contains an element
$s$ such that every element of $G$ can be defined by a first order
formula with parameter $s$. The aim of this paper is to investigate the structure of such groups.
\end{abstract}

\maketitle


\vspace{2cm}

Let $\mathcal{L}=(\cdot, ^{-1}, 1)$ be the language of groups.
Suppose $G$ is a group and $S\subseteq G$. We extend $\mathcal{L}$
to a new language $\mathcal{L}(S)$, by attaching new constant
symbols $a_s$, for any $s\in S$. So we have
$\mathcal{L}(S)=\mathcal{L}\cup\{ a_s: s\in S\}$. Then, $G$ becomes
an $\mathcal{L}(S)$-structure if we let $s$ to be the interpretation of
$a_s$ in $G$. Now, suppose $g\in G$ is an arbitrary element and
there exists a first order formula $\varphi(x)$ in the language
$\mathcal{L}(S)$ (with a free variable $x$), such that
$$
\{ g\}=\{ x\in G:\ G\vDash \varphi(x)\}.
$$
Then we say that $g$ is {\em definable} by the elements of $S$ or
$S$-{\em definable} for short. Let $\mathrm{def}_S(G)$ be the set of
all $S$-definable elements of $G$. Clearly this subset is a
subgroup. If we have $\mathrm{def}_S(G)=G$, then we say that $S$
{\em logically generates} $G$. A {\em logically cyclic} group is a
group which is logically generated by a single element. A cyclic
group is then clearly logically cyclic as every element can be
defined as a power of the generator. The converse is not true, for
example the additive group of rationales, $G=(\mathbb{Q}, +)$, is
logically cyclic as every element $g=m/n$ can be defined by the
formula $nx=m$, which is clearly a first order formula having $s=1$
as a parameter. It is easy to see that all subgroups of $G=(\mathbb{Q}, +)$ are logically cyclic.
In this paper , we  prove that if  a finite group
is logically cyclic, then it is cyclic in the ordinary sense. We also determine all finitely generated
logically cyclic groups as well as those logically cyclic groups  which are torsion-free.

\section{Preliminaries}
We can use the definability theorem of Svenonius to study definable
elements in  groups. For the case of finite groups, a version of this
theorem will be used which can be proved by an elementary argument.
We briefly discuss this well-known result of the model theory. Let
$\mathcal{L}$ be a first order language and $M$ be a structure in
$\mathcal{L}$. An $n$-ary relation $R\subseteq M^n$ is said to be
definable, if there exists a formula $\varphi(x_1, \ldots, x_n)$ in
the language $\mathcal{L}$, such that
$$
R=\{ (x_1, \ldots, x_n)\in M^n:\ M\vDash \varphi(x_1, \ldots,
x_n)\}.
$$
It is easy to see that if $R$ is definable, then every automorphism
of $M$ preserves $R$. To see what can be said about the converse, we
need the concept of {\em elementary extension}. For any
$\mathcal{L}$-structure $M$, suppose $\mathrm{Th}(M)$ is the first
order theory of $M$, i.e. the set of all first order sentences,
which are true in $M$. We say that an $\mathcal{L}$-structure
$M^{\prime}$ is an elementary extension of $M$, if $M$ is a
sub-structure of $M^{\prime}$ and $\mathrm{Th}(M)\subseteq
\mathrm{Th}(M^{\prime})$. We are now ready to review the theorem of
Svenonius.

\begin{theorem}
A relation $R\subseteq M^n$ is definable if and only if every
automorphism of every elementary extension of $M$ preserves $R$.
\end{theorem}

For a proof, the reader can see \cite{Poizat}. Suppose we want to
use this theorem in the case of groups; we must assume that $G$ is a
group, $S\subseteq G$ is an arbitrary subset, and $\mathcal{L}(S)$
is the extended language of groups with parameters from $S$. Clearly
a singleton set $\{ g\}$ is a unary relation in $G$, so we can
restate the above theorem, for $S$-definability of elements in $G$.

\begin{corollary}
An element $g$ is $S$-definable in $G$ if and only if,  for any
elementary extension $G^{\prime}$ of $G$ and every automorphism
$\alpha:G^{\prime}\to G^{\prime}$, if $\alpha$ fixes elements of
$S$, then it fixes also $g$.
\end{corollary}

If $G$ is a finite group, then the only elementary extension of $G$
is $G$ itself, because there exists a first order sentence which
says that $G$ has $m$ elements ($m$ is the order of $G$), so any
elementary extension of $G$ must have order $m$. Hence, for the case
of finite groups, we have

\begin{corollary}
An element $g$ is $S$-definable in a finite group $G$ if and only
if, for any  automorphism $\alpha:G\to G$, if $\alpha$ fixes
elements of $S$, then it fixes also $g$.
\end{corollary}

Note that this very special case of Svenonius's theorem can be
proved independently by an elementary argument. Here is the proof.

\begin{proof} Let $g\in G$ be invariant under every automorphism
which fixes elements of $S$. Let
$$
G=\{ g_1, g_2, \ldots, g_n\}
$$
be an enumeration of the elements of $G$ such that $g_1=1$, $\{g_2,
\ldots, g_m\}=S$ and $g_{m+1}=g$. Consider the Cayley table of $G$,
i.e. determine the unique numbers $\sigma(i,j)$ such that
$g_ig_j=g_{\sigma(i,j)}$. Now we introduce a formula $\varphi(y)$ in
the language $\mathcal{L}(S)$ as
\begin{eqnarray*}
\exists x_1, \ldots, x_n&:& (\bigwedge_{i\neq j}x_i\neq
x_j)\wedge((\bigwedge_{i=2}^mx_i=g_i)\wedge x_1=1)\wedge
(\bigwedge_{i,j}x_ix_j=x_{\sigma(i,j)})\\
&\ & \wedge(y=x_{m+1}).
\end{eqnarray*}

We show that $\varphi(y)$ defines $g$. Let $a\in G$ be an element
such that $G\models \varphi(a)$. Therefore there exist distinct
elements $b_1, \ldots, b_n\in G$ such that\\

1- $b_1=1$ and $b_2=g_2, \ldots, b_m=g_m$.

2- $b_{m+1}=a$.

3- $b_ib_j=b_{\sigma(i,j)}$.\\

Now, the map $f:G\to G$ defined by $f(g_i)=b_i$ is an automorphism,
and for all $g_i\in S$ we have $f(g_i)=b_i=g_i$. So, $f$ must
preserve $g$. Hence, we have $g=f(g)=f(g_{m+1})=b_{m+1}=a$. This
completes the proof.
\end{proof}

An automorphism $\alpha:G\to G$ is said to be an $S$-automorphism,
if it fixes $S$ elementwise. If we work in the semidirect product $\hat{G}=\mathrm{Aut}(G)\ltimes G$,
then the set of all $S$-automorphisms of $G$ is just the centralizer $C_{\mathrm{Aut}(G)}(S)$.
We will use this type of centralizer notation in the rest of the article.
So, for an arbitrary group, we have
$$
\mathrm{def}_S(G)\subseteq C_G(A),
$$
where $A=C_{\mathrm{Aut}(G)}(S)$.
If $G$ is finite, then by the corollary 1.3, we have the equality
$$
\mathrm{def}_S(G)= C_G(A).
$$

Suppose $G$ is logically cyclic. So $G=\mathrm{def}_s(G)$ for some
$s$. This shows that $G=C_G(A)$, therefore we have the implication
$$
\forall \alpha\in \mathrm{Aut}(G):\ \alpha(s)=s\Rightarrow
\alpha=\mathrm{id}_G.
$$
For finite groups, 1.3 implies that the converse is also true, i.e.
if there exists an element $s$ satisfying the above implication,
then $G$ is logically cyclic. We  prove that every logically cyclic
group is abelian. Note that a similar argument shows that for any
group $G$ and every element $s\in G$, the subgroup
$\mathrm{def}_s(G)$ is also abelian (see also the discussion at the end of the next section). It is also possible to prove that, if a group $G$ is logically generated by a commuting set of elements $S$, then it is Abelian.

\begin{proposition}
Every logically cyclic group is abelian.
\end{proposition}

\begin{proof}
Let $G=\mathrm{def}_s(G)$. Then as above
$$
\forall \alpha\in \mathrm{Aut}(G):\ \alpha(s)=s\Rightarrow
\alpha=\mathrm{id}_G.
$$
So, considering the inner automorphism $\alpha: x\mapsto sxs^{-1}$,
we obtain $s\in Z(G)$. Now, let $g\in G$ be an arbitrary element and
let $\beta: x\mapsto gxg^{-1}$. Since $[g, s]=1$, so $\beta(s)=s$
and therefore $\beta=\mathrm{id}_G$. This shows that $g\in Z(G)$ and
hence $G$ is abelian.
\end{proof}

One may ask this question: If a group $G$ is logically generated by a set $S$ and $\langle S\rangle$ is nilpotent, can we prove that $G$ is nilpotent? \\
Note that if every element of $G$ is definable in the language of
groups, $\mathcal{L}$, then we must have
$$
G=C_G(\mathrm{Aut}(G)),
$$
and in this case we have $\mathrm{Aut}(G)=1$, which shows that $G=1,
\mathbb{Z}_2$. So, the groups $1$ and $\mathbb{Z}_2$ are the only
groups, every element in which, is definable in the language of
groups.

As we saw, if a  group $G$, is
logically cyclic, then there exists an element $s\in G$
such that for all non-identity automorphism $\alpha$, we have
$\alpha(s)\neq s$. In this case, if we consider the action of
$\mathrm{Aut}(G)$ on $G$, then
$$
|\mathrm{Orb}(s)|=|\mathrm{Aut}(G)|,
$$
so, for  logically cyclic groups, we have
$$
|\mathrm{Aut}(G)|\leq |G|.
$$

\section{The case of finite groups}
We are now ready to prove one of our main theorems.

\begin{theorem}
Let $G$ be a finite logically cyclic group. Then $G$ is cyclic.
\end{theorem}

\begin{proof}
As we said before, $G$ is abelian and so it is a direct product of
abelian $p$-groups of the form
$$
H_p=\mathbb{Z}_{p^{e_1}}\times \cdots\times\mathbb{Z}_{p^{e_t}},
$$
where $p$ is a prime (ranging in the set of all prime divisors of
$|G|$) and $1\leq e_1\leq \cdots\leq e_t$ are depending on $p$. Note
that if a finite group $G=G_1\times G_2$ is logically cyclic, then
both $G_1$ and $G_2$ are logically cyclic. This is because, if for
example $G_1$ is not logically cyclic, then (by 1.3) for all $s_1\in
G_1$, there exists a non-identity automorphism $\varphi_1\in
\mathrm{Aut}(G_1)$, such that $\varphi_1(s_1)=s_1$. Hence for all
$(s_1, s_2)\in G$, there exists the non-identity $(\varphi_1,
\mathrm{id}_{G_2})\in \mathrm{Aut}(G)$, such that
$$
(\varphi_1, \mathrm{id}_{G_2})(s_1, s_2)=(s_1, s_2),
$$
and this violates the logically cyclicity of $G$. The converse is
also true if the orders of $G_1$ and $G_2$ are co-prime, since, in
this case we have
$$
\mathrm{Aut}(G)=\mathrm{Aut}(G_1)\times \mathrm{Aut}(G_2).
$$

This argument shows that it is enough to assume that $G$ has the
form
$$
\mathbb{Z}_{p^{e_1}}\times \cdots\times\mathbb{Z}_{p^{e_t}}.
$$
By \cite{Hillar}, the order of $\mathrm{Aut}(G)$ can be computed as
follows. Let
$$
d_i=\max\{ j:\ e_j=e_i\}, \ c_i=\min\{ j:\ e_j=e_i\}.
$$
Then we have
$$
|\mathrm{Aut}(G)|=\prod_{i=1}^t(p^{d_i}-p^{i-1})p^{e_i(t-d_i)+(e_i-1)(t-c_i+1)}.
$$
Suppose $t\geq 2$. Since $G$ is logically cyclic, so by the above
observation, $A=\mathbb{Z}_{p^{e_1}}\times\mathbb{Z}_{p^{e_2}}$ is
logically cyclic. We compute the order of $\mathrm{Aut}(A)$, using
the above formula. Note that, in the case of the group $A$, we have
$$
1\leq d_1\leq 2,\ d_2=2,\ c_1=1,\ 1\leq c_2\leq 2.
$$
So, we have
$$
|\mathrm{Aut}(A)|=(p^{d_1}-1)(p-1)p^{4e_1+3e_2-d_1e_1-c_2e_2+c_2-4}.
$$
Applying the requirement $|\mathrm{Aut}(A)|\leq |A|$, we obtain
$$
(p^{d_1}-1)(p-1)p^{(3-d_1)e_1+(2-c_2)e_2+c_2}\leq p^4,
$$
so we can consider some possibilities for $d_1$ and $d_2$.\\

1- First, note that the case $d_1=1$ and $c_2=1$ is impossible.\\

2- If $d_1=1$ and $c_2=2$, then we have
$$
(p^1-1)(p-1)p^{3e_1+2e_2-e_1-2e_2+2}\leq p^4,
$$
and hence
$$
(p-1)^2p^{2e_1+2}\leq p^4.
$$
Now the case $e_1>1$ is impossible and hence $e_1=1$. This shows
that $p=2$ and hence $A=\mathbb{Z}_2\times \mathbb{Z}_{2^f}$ for
some $f\geq 2$. It is easy to see that $|\mathrm{Aut}(
\mathbb{Z}_2\times \mathbb{Z}_{2^f})|=2^{f+1}$ and therefore the
whole the  group $A=\mathbb{Z}_2\times \mathbb{Z}_{2^f}$ must be the
orbit of $s$ under the action of its automorphism group, which is
not the case.
So, we get a contradiction.\\

3- Let $d_1=2$ and $c_2=1$. This shows that $e_1=e_2$, and hence
$$
(p^2-1)(p-1)p^{2e_1+1}\leq p^4.
$$
Again, the case $e_1>1$ is impossible and the case $e_1=1$ implies
$(p^2-1)(p-1)\leq p$ which is contradiction.\\

4- Finally, note that the case $d_1=2$ and $c_2=2$ is also
impossible.\\

This argument shows that in all the cases,  $t\geq 2$ is not valid.
So $G$ is a direct product of cyclic groups of co-prime orders and
hence it is cyclic.

\end{proof}

A caution is necessary here: the subgroup $\mathrm{def}_s(G)$ is
strongly dependent to $G$. If we are not careful about this
dependence, we may obtain wrong conclusions. As an example, let
$H=\mathrm{def}_s(G)$. Since every element of $H$ is definable by
the parameter $s$, so one may concludes that $H$ is logically
cyclic. In the other words, one may convince that by 2.1, for any
finite group $G$ and any $s\in G$ the group
$C_G(C_{\mathrm{Aut}(G)}(s))$ is cyclic. This is not true, since for
example, if we let $G$ be a $p$-group of class 2 with an odd $p$,
and if we assume that  $\mathrm{Aut}(G)$ is also $p$-group such that
$\Omega_1(G)$ is not included in the center, then we can choose $s$
to be a non-central element of order $p$ and $u\in
C_{\mathrm{Aut}(G)}(G)\cap Z(G)$ with order $p$. Now, it is easy to
see that $\langle s, u\rangle\subseteq C_G(C_{\mathrm{Aut}(G)}(s))$,
and so this subgroup is not cyclic. Note that for the case $p=2$,
the dihedral group of order 8 is also a counterexample. These
counterexamples show that in general $\mathrm{def}_s(G)$ is not
logically cyclic. To see the reason, note that if $H\leq G$ and
$s\in H$, then there is no trivial relations between
$\mathrm{def}_s(G)$ and $\mathrm{def}_s(H)$. If $g\in
\mathrm{def}_s(H)$, and $\varphi(x)$ is a formula defining $g$ in
$H$, then we may have
$$
|\{ a\in G:\ G\vDash \varphi(a)\}|>1,
$$
or even, we may have $G\vDash \neg \varphi(a)$. This shows that in
general $\mathrm{def}_s(H)$ is not contained in $\mathrm{def}_s(G)$.
On the other hand, if $g\in H\cap \mathrm{def}_s(G)$, then we may
have not $g\in \mathrm{def}_s(H)$ by a similar argument. Hence, the
subgroup $\mathrm{def}_s(G)$ behaves not so simply despite its
abelianity. In some cases, the subgroup $\mathrm{def}_s(G)$ is also logically cyclic, for example if $G$ is a divisible Abelian group. To see this, one may use the quantifier elimination property of divisible Abelian group and a similar argument as in the proof of Theorem 3.3 below.

\section{The case of infinite groups}
In this section, we will determine the structure of logically cyclic groups for the following cases:\\

1- Finitely generated groups,

2- Divisible groups,

3- Torsion-free groups.\\

Suppose $G=\mathrm{def}_s(G)$ is a logically cyclic group and $H=\langle s\rangle$. Then as we saw in the introduction, the subgroup $H$ is an $\mathrm{Aut}$-basis of $G$ in the sense of \cite{Gio} and \cite{Gio2}. This means that every automorphism of $G$ is uniquely determined by its action on $H$. So, as it is proved in \cite{Gio}, we have
$$
\mathrm{Hom}(G/H, H)=0.
$$
Applying results of \cite{Gio2} concerning finite $\mathrm{Aut}$-bases, we can collect the following facts about the group $G$.\\

1- If $G$ is periodic, then it is finite so by the previous section it is cyclic.

2- If $G$ is periodic by finitely generated, then it is finitely generated.

3- $\mathrm{Aut}(G)$ is countable. Note that this is also a result of the inequality $|\mathrm{Aut}(G)|\leq |G|$, which is valid for all logically cyclic groups.

4- If the order of $s$ is finite then $G$ is finite and so it is cyclic.

5- The quasi-cyclic groups $\mathbb{Z}_{p^{\infty}}$ are not logically cyclic as well as the additive group $\mathbb{Q}/\mathbb{Z}$.

We first determine the structure of all finitely generated logically cyclic groups.

\begin{proposition}
Let $G$ be a logically cyclic finitely generated group. Then $G$ is cyclic or $G=\mathbb{Z}\times \mathbb{Z}_2$.
\end{proposition}

\begin{proof}
We have $G=\mathbb{Z}^n\oplus A$ for some finite group $A$. It is easy to apply Corollary 1.3 to see that $A$ is also logically cyclic, so by the previous section $A=\mathbb{Z}_m$, for some $m$. Now, suppose $n>1$ and $H=\langle s\rangle$. Since $s$ has infinite order by the fact 4, we have $G/H=\mathbb{Z}^{n-1}\oplus A$ and hence there exists a non-zero homomorphism $G/H \to H$, contradicting the fact $\mathrm{Hom}(G/H, H)=0$. Therefore  $n\leq 1$. Now, suppose that $G$ is not cyclic. We show that $m=2$. Note that the group $G=\mathbb{Z}\times \mathbb{Z}_2$ is actually a logically cyclic group. To see this, we  show that $G$ is logically generated by $s=(1,0)$. Clearly all elements of the form $(u,0)$ can be defined by $x=us$, so consider an element of the form $(u,1)$. Since we have $(u,1)=us+(0,1)$ and $(0,1)$ is the only element of order 2 in the whole group, so we can define $(u,1)$ by the formula
$$
     \forall y\ ((2y=0 \wedge y\neq 0) \Rightarrow x=us+y).
$$
Suppose now, $m\geq 3$ and $s=(u,v)$ is a logical generator of $G$. Note that $s$ can not be of the form $(u,0)$, because in this case we can fix a non-identity automorphism $\alpha\in \mathrm{Aut}(\mathbb{Z}_m)$ (as we let $m\geq 3$) and then the non-identity automorphism $(\mathrm{id}, \alpha)$ will fix $s$, which is impossible. Also it is impossible to have $s=(0,v)$, since in this case $o(s)$ is finite and this implies that $G$ is also finite by the fact 4 above. 

Recall that every endomorphism  $f:G\to G$ can be represented as a matrix 
$$
M=\left [
\begin{array}{cc}
 \lambda_n& 0\\
               \gamma_b& \eta_a
             \end{array}
  \right ], 
$$
where $\lambda_n:\mathbb{Z}\to \mathbb{Z}$ is defined by $\lambda_n(x)=nx$, $\gamma_b:\mathbb{Z}\to \mathbb{Z}_m$ is defined by $\gamma_b(x)=bx\ \ (\mathrm{mod}\ m)$, and $\eta_a:\mathbb{Z}_m\to \mathbb{Z}_m$ is defined by $\eta_a(x)=ax\ \ (\mathrm{mod}\ m)$. We know that this is an automorphism iff the matrix $M$ is invertible and this happens just in the case $(a, m)=1$. Note that also $M$ represents the identity iff $a=1$ and $m$ divides $b$. We first investigate the case when $m$ and $u$ are co-prime. Choose $a\neq 1$ co-prime to $m$ (this is possible as we assumed that $m\geq 3$). Then there is an integer $b$ such that
$$
bu+(a-1)v \equiv 0\ \ (\mathrm{mod}\ m).
$$
So, consider the automorphism 
$$
M=\left [
\begin{array}{cc}
 \mathrm{id}& 0\\
               \gamma_b& \eta_a
             \end{array}
  \right ].
$$
We have
$$
Ms=\left [
\begin{array}{cc}
 \lambda_n& 0\\
               \gamma_b& \eta_a
             \end{array}
\right ]
\left [
\begin{array}{c}
u\\
v
\end{array}
\right ]
=
\left [
\begin{array}{c}
u\\
bu+av
\end{array}
\right ]
=
\left [
\begin{array}{c}
u\\
v
\end{array}
\right ]=s.
$$
This shows that $s$ can not be a logical generator of $G$. So, we have $d=(m, u)>1$. Now, put $a=1$ and $b=m/d$ and consider again the automorphism $M$. This is a non-identity automorphism as $m$ does not divide $m/d$. It is easy now to see $Ms=s$, a contradiction.
\end{proof}

We can use a similar argument as above to show that if $G$ is a torsion-free logically cyclic group, then so is $G\times \mathbb{Z}_2$. To see this, 
first we show that $\mathbb{Q}\times \mathbb{Z}_2$ is  logically generated by $s=(1,0)$. Clearly every element of the form $(m/n, 0)$ can be defined by $nx=ms$, so consider the element $(m/n, 1)$. This element can be defined by
$$
\forall y \forall z \ ((2y=0 \wedge y\neq 0 \wedge nz=ms) \Rightarrow x=y+z).
$$
Now, we can apply this observation and  the theorem 3.3 below to prove the general case.\\

The next proposition, shows that the additive group of rationals is the only divisible logically cyclic group.

\begin{proposition}
Let $G$ be a non-trivial divisible logically cyclic group. Then $G=(\mathbb{Q}, +)$.
\end{proposition}

\begin{proof}
By a well-know theorem on divisible abelian groups, $G=\mathbb{Q}^I\oplus \sum_{p\in J}\oplus\mathbb{Z}_{p^{\infty}}$, where $I$ is a set and $J$ is a set of primes. Let $J\neq \emptyset$. Then for some prime $p$ the quasi-cyclic group $\mathbb{Z}_{p^{\infty}}$ is a direct summand of $G$ and hence $\mathrm{Aut}(\mathbb{Z}_{p^{\infty}})$ embeds in $\mathrm{Aut}(G)$. But by \cite{Gio2}, the cardinality of the automorphism group of the quasi-cyclic group is uncountable while $|\mathrm{Aut}(G)|\leq |G|$ is countable. Therefore, $J=\emptyset$ and so $G= \mathbb{Q}^I$. Let $|I|>1$. Then for any $0\neq s\in \mathbb{Q}^I$, we can construct a non-trivial automorphism $\alpha\in \mathrm{GL}_I(\mathbb{Q})$ such that $\alpha(s)=s$. But, this contradicts the assumption of logical cyclicty of $G$.
\end{proof}

Finally, we give a characterization of logically cyclic torsion-free groups. In the proof, we use the well-known quantifier elimination property of divisible groups, which says that in such a group, every first order formula is equivalent to a quantifier-free one.

\begin{theorem}
Let $G=\mathrm{def}_s(G)$ be a torsion-free logically cyclic group. Then $G$ embeds in $(\mathbb{Q}, +)$.
\end{theorem}

\begin{proof}
Let $G^{\ast}$ be the divisible envelope of $G$, so $G$ is an essential subgroup of $G^{\ast}$, i.e. for any $0\neq u\in G^{\ast}$ the intersection $G\cap \langle u\rangle$ is non-trivial. We prove that $G^{\ast}$ is also logically cyclic and $s$ is its logical generator. Suppose $0\neq u\in G^{\ast}$. There is a non-zero integer $m$ such that $0\neq mu=v\in G$. Let $\varphi(x)$ be a formula which defines $v$ in $G$. Since $G^{\ast}$ is divisible, so it has the quantifier elimination property. Hence in $G^{\ast}$, the formula $\varphi(x)$ is equivalent to a quantifier-free formula $\psi(x, {\bf a})$, where ${\bf a}$ is a set of elements of $G^{\ast}$. Note that $\psi(x, {\bf a})$ is a Boolean combination of atomic formulae of the form $m_{ij}x=a_{ij}$ with $m_{ij}\in \mathbb{Z}$ and $a_{ij}\in G^{\ast}$, so
$$
\psi(x, {\bf a})\equiv \bigvee_{i=1}^p\bigwedge_{j=1}^{q_i}(m_{ij}x=a_{ij})^{\pm},
$$
where $\pm$ indicates an atomic formula or a negation of an atomic formula. Since $v$ is a solution $\psi(x, {\bf a})$, so there is an index $i$ such that we have
$$
\bigwedge_{j=1}^{q_i}(m_{ij}v=a_{ij})^{\pm}.
$$
If all conjunctives in the recent formula are negative, then there will be infinitely many solutions for it in $G$, which is not the case. So, there is a $j$ such that $m_{ij}v=a_{ij}$. Since $G$ is torsion-free, so we conclude that $v$ is defined by $m_{ij}v=a_{ij}$ in $G^{\ast}$, i.e.
$$
\varphi(x)\equiv (m_{ij}v=a_{ij}).
$$
Now, consider the following formula in the language of groups with parameter $s$,
$$
\forall x\ (\varphi(x)\Rightarrow my=x).
$$
Clearly, this formula, defines $u$ in $G^{\ast}$ and so, $G^{\ast}$ is logically cyclic. By the previous proposition, $G^{\ast}=\mathbb{Q}$, and the proof completes.
\end{proof}

One more problem remains unsolved;  the classification of logically cyclic  algebraic structures
other than groups. An algebra $A$ in an algebraic language
$\mathcal{L}$, is said to be cyclic if it is  generated by a single
element. It is called logically cyclic if every element of $A$ can
be defined by a first order formula containing a fixed parameter
from $A$. If $A$ is finite, then clearly $A$ is the only elementary
extension of itself. So an element $a\in A$ is definable using a
parameter $s$, if and only if every automorphism of $A$ which fixes
$s$, fixes already $a$. How can we obtain the relation between
cyclic and logically cyclic algebras? This may need further efforts
because in general we have a few information about
$\mathrm{Aut}(A)$.\\

{\bf Acknowledgement} The author would like to thank J. S. Eyvazloo,  G. Robinson and K. Bou-Rabee for their comments and suggestions.


\begin{thebibliography}{99}

\bibitem{Gio} C. Giovanni, N. Chiara, {\em A note on endomorphisms
of hypercentral groups}, J. Algebra, Vol. 225, No. 1, (2002).

\bibitem{Gio2} C. Giovanni, N. Chiara, {\em Subgroups defining automorphisms in
locally nilpotent groups}, Forum Math., Vol. 15, No. 4, (2003).

\bibitem{Hillar} C. J. Hillar, D. L. Rhea, {\em Automorphisms of
finite abelian groups}, American Math. Monthly, Vol. 114, No. 10,
(2007).

\bibitem{Poizat} B. Poizat, {\em A Course in Model Theory: an introduction to contemporary mathematical logic}, Springer, (2000).


\end{thebibliography}
\end{document}